\documentclass[a4paper,11pt]{scrartcl}
\usepackage[utf8]{inputenc}
\usepackage{amsmath,amssymb,amsthm}
\usepackage{enumerate}
\usepackage{url}
\newtheorem{thm}{Theorem}[section]
\newtheorem{dfn}[thm]{Definition}
\newtheorem{lem}[thm]{Lemma}
\newtheorem{prop}[thm]{Proposition}
\newtheorem{cor}[thm]{Corollary}

\theoremstyle{remark}
\newtheorem{rem}[thm]{Remark}

\DeclareMathOperator{\diam}{diam}
\DeclareMathOperator{\vol}{vol}

\newcommand{\ddt}[1]{\frac{\partial #1 }{\partial t}}

\newcommand{\ph}{\varphi}
\title{Canonical smoothing of compact Alexandrov surfaces via Ricci flow.}
\author{Thomas Richard\footnote{UJF-Grenoble I, Institut Fourier,
    Grenoble, F-38402, France}\ \footnote{
    CNRS UMR5882, Institut Fourier, Grenoble, F-38041, France}}
\begin{document}
\maketitle
\begin{abstract}
  In this paper, we show existence and uniqueness of Ricci
  flow whose initial condition is a compact Alexandrov surface with
  curvature bounded from below. This requires a weakening of the
  notion of initial condition which is able to deal with a priori
  non-Riemannian metric spaces. As a by-product, we obtain that the
  Ricci flow of a surface depends smoothly on Gromov-Hausdorff
  perturbations of the initial condition.
\end{abstract}

\section*{Introduction}
Ricci flow of smooth manifolds has had strong applications to the
study of smooth Riemannian manifolds. It is therefore natural to ask
if Ricci flow can be helpful in the study non-smooth geometric objects. A
reasonable assumption to make on a metric space $(X,d)$ that we want
to deform by the Ricci flow is to require $(X,d)$ to be approximated
in some sense by a sequence $(M_i,g_i)$ of smooth Riemannian
manifolds. In \cite{MR2526789} and \cite{MS2009}, M. Simon studied a
class of 3-dimensional metric 
spaces by this method.  An important feature of such ``Ricci flows of
metric spaces'' is that the notion of initial condition has to be
weakened. In the work of M. Simon \cite{MR2526789} and \cite{MS2009},
and of the author \cite{2011arXiv1111.0859R}, a
weak notion of inititial condition has been used, which we call
``metric initial condition'' :
\begin{dfn}
  A Ricci flow $(M,g(t))_{t\in (0,T)}$ on a compact manifold $M$ is
  said to have the metric 
  space $(X,d)$ as metric initial condition if the Riemannian
  distances $d_{g(t)}$ uniformly converge as $t$ goes to $0$ (as
  functions $M\times M\to 
  \mathbb{R}$) to a distance $\tilde{d}$  
  on $M$ such that $(M,\tilde{d})$ is isometric to $(X,d)$.
\end{dfn}
\begin{rem}
  The compactness assumption in the definition gives that $(X,d)$ is
  homeomorphic to $M$ with its manifold topology. This follows from
  the fact that $\tilde{d}$ is continuous on $M$, which implies that
  the identity of $M$ is continuous as an application from $M$ with
  its usual topology to $M$ with the topology definied by $\tilde{d}$,
  compactness of $M$ then give that the identity is an homeomorphism.
\end{rem}

The existence of such flows for some classes of metric spaces $(X,d)$
has been proved in \cite{MR2526789},\cite{MS2009} and
\cite{2011arXiv1111.0859R}. An interesting class of spaces for which 
existence holds is the class of compact Alexandrov surfaces whose
curvature is bounded from below. In this paper we prove uniqueness for
the Ricci flow with such surfaces as metric initial condition, more
precisely :
\begin{thm}\label{uniqthm}
  Let $(M_1,g_1(t))_{t\in(0,T]}$ and $(M_2,g_2(t))_{t\in(0,T]}$ be two
  smooth Ricci flows which admit a compact Alexandrov surface $(X,d)$ as
  metric initial condition. Assume furthermore that one can find $K>0$ such that :
\[ \forall (x,t)\in M_i\times (0,T]\quad K_{g_i(t)}(x)\geq -K.\]

Then there exist a conformal diffeomorphism $\ph:M_1\to M_2$ such that
$g_2(t)=\ph^* g_1(t)$.
\end{thm}
Note that the required bounds on the Ricci flow are those which are
provided by the existence proof.

In the next few lines, we outline the proof of Theorem
\ref{uniqthm}. Any of the two Ricci flows $(M_i,g_i(t))$ stays in a
fixed conformal class, and thus can be written $g_i(t)=w_i(x,t)h_i(x)$
for some fixed background metric $h_i$ which can be chosen to have
constant curvature. We fisrt show that the metric initial condition
prescribes the conformal class of the flow, thus we can assume that
$h_1=h_2=h$. The proof of this fact uses deep results from the theory
of singular surfaces introduced by A. D. Alexandrov. This implies that
our two Ricci flows can be seen as solutions of the
following nonlinear PDE on $(M,h)$ :
\[\ddt{w_i}=\Delta_h \log(w_i)-2K_h.\]
One then shows that each of the $w_i$ has an $L^1$ initial condition
as $t$ goes to $0$ and uses standard techniques to show uniqueness.

Our result can be stated in two other ways :
\begin{prop}
  Let $M$ be a smooth compact topological surface, and $d$ be a
  distance on such that $(M,d)$ is an Alexandrov surface with
  curvature bounded from below.

  Let $g_1(t)_{t\in(0,T)}$ and $g_1(t)_{t\in(0,T)}$ be two Ricci flows
  on $M$ which are smooth with respect to some differential structures
  on $M$. Assume furthermore that one can find $K>0$ such that :
  \[ \forall (x,t)\in M_i\times (0,T]\quad K_{g_i(t)}(x)\geq -K\]
 and that for $i=1,2$ the distances $d_{g_i(t)}$ uniformly converge to
 $d$ as $t$ goes to $0$.
 
 Then $g_1(t)=g_2(t)$ for $t\in(0,T)$.
\end{prop}
This proposition is not a consequence of Theorem \ref{uniqthm}, but
just requires a minor adjustment in its proof, which will be indicated
in Section \ref{sec:uniq-conf-class}.
\begin{prop}
  Let $(M_1,g_1(t))_{t\in(0,T]}$ and $(M_2,g_2(t))_{t\in(0,T]}$ be two
  smooth Ricci flows such that for $i=1,2$ $(M_i,g_i(t))$
  Gromov-Hausdorff converges to a compact Alexandrov surface $(X,d)$
  whith curvature bounded from below as $t$ goes to $0$. Assume furthermore that one can find $K>0$ such that :
  \[ \forall (x,t)\in M_i\times (0,T]\quad K_{g_i(t)}(x)\geq -K.\]

  Then there exist a conformal diffeomorphism $\ph:M_1\to M_2$ such that
  $g_2(t)=\ph^* g_1(t)$.
\end{prop}
\begin{proof}
  We just have to show that if $(M^2,g(t))_{t\in(0,T)}$ is a smooth
  Ricci flow on a surface $M^2$ such that for all $t\in(0,T)$
  $K_{g(t)}\geq -K$ and such that $(M^2,g(t))$ Gromov-Hausdorff
  converges to $(X,d)$ as $t$ goes to $0$, then $(X,d)$ is the metric
  initial condition for the Ricci flow $(M^2,g(t))$.

  Since the diameter and the volume are continuous with respect to
  Gromov-Hausdorff convergence with sectionnal curvature bounded from
  below, we have bounds on the diameter and the volume of $(M,g(t))$
  which are independent of $t$. Thanks to the lower bound on the
  curvature and Bushop-Gromov inequality, we thus have some $v_0>0$
  such that :
  \[\forall t\in(0,T)\ \forall x\in M\quad
  \vol_{g(t)}(B_{g(t)}(x,1))\geq v_0.\]
  Thanks to Lemma 4.2 in \cite{MS2009}, we then have that, for some
  constant $C>0$ and all
  $t\in(0,T)$ (for some possibly smaller $T>0$) :
  \[\forall t\in(0,T)\quad |K_{g(t)}|\leq\frac{C}{T}.\]
  One can then argue as in the proof of Theorem 9.2 of \cite{MS2009}
  to show that, as $t$ goes $0$, the Riemannian distances uniformly
  converge to a distance $\tilde{d}$ on $M$ such that $(M,\tilde{d})$
  is isometric to $(X,d)$. Thus $(X,d)$ is the metric initial
  condition of the Ricci flow $(M,g(t))$.
\end{proof}
As a corollary, we obtain the following statement, which says that for
surfaces with curvature bounded from below Gromov-Hausdorff
convergence of the initial conditions implies smooth convergence of the Ricci flows :
\begin{cor}\label{corcont}
  Let $(M_i,g_i)_{i\in\mathbb{N}}$ be a sequence of compact surfaces
  with curvature bounded from below which converges to a compact Alexandrov
  surface $(X,d)$ with curvature bounded from below, then there exist $T>0$
  such that the Ricci flows $(M_i,g_i(t)) _{i\in\mathbb{N}}$ with initial condition
  $(M_i,g_i)$ exist at least for $t\in [0,T)$ and converges (as smooth
  Ricci flows on $(0,T)$) to the unique Ricci flow with metric initial
  condition satisfying the bounds of Theorem \ref{uniqthm}.
\end{cor}
\begin{proof}[Proof of Corollary \ref{corcont}]
  Let $(M_i,g_i)_{i\in\mathbb{N}}$ be sequence satisfying the
  assumptions of Corollary \ref{corcont}. By continuity of the volume
  and the diameter with respect to Gromov-Hausdorff convergence of
  Alexandrov surfaces, we have constants $V$ and $D$ such that for any
  $i\in\mathbb{N}$ :
  \begin{itemize}  
  \item $K_{g_i}\geq -1$
  \item $\diam(M_i,g_i)\leq D$
  \item $\frac{V}{2}\leq \vol(M_i,g_i)\leq V$
  \end{itemize}
  The existence theory (Theorem \ref{estiexist}) implies the Ricci
  flows $(M_i,g_i(t))$ exist at least for $t\in [0,T)$ and form a
  precompact sequence whose accumulation points can only be Ricci flows with metric
  initial condition $(X,d)$ satifying the bounds of Theorem
  \ref{uniqthm}. The uniqueness theorem then implies that there is
  only one accumulation point.
\end{proof}

Uniqueness and non-uniqueness issues have been previously considered
for the Ricci flow of surfaces with ``exotic'' initial conditions in
the works of Giesen and Topping
(\cite{MR2832165},\cite{2010arXiv1010.2795T}) and Ramos \cite{2011arXiv1109.5554R}.

The paper is organised as follows, in the first section, we sketch
M. Simon's existence proof in dimension 2. In the second section, we
show that the metric initial condition uniquely specifies the
conformal class. The last section completes the proof. In the
appendix, we quickly summarize the results we need from the theory of
Alexandrov surfaces.

\subsection*{Acknowledgements}
\label{sec:acknowledgements}

The author wishes to thank V. Kapovitch and M. Troyanov for answering
his questions on Alexandrov surfaces. He also thanks his advisor
G. Besson for
his guidance and support.
\section{Existence}

Here we briefly review the work of Miles Simon which shows the
existence of a Ricci flow for compact Alexandrov surfaces with
curvature bounded from below. Without loss of generality, we will
assume that all Alexandrov surfaces with curvature bounded from below
have curvature bounded from below by $-1$.

What allows us to flow these surfaces is that they can be approximated
by smooth surfaces in a controlled way. This is what Theorem
\ref{approxalex} in the appendix says.

We will now construct a Ricci flow with metric initial condition
$(X,d)$ as limit of the Ricci flows of the $(M_i,g_i)$. In order to do
this, we use the following estimates due to M. Simon :

\begin{thm}\label{estiexist}
  For any $V>0$ and $D>0$, there exists $\kappa>0$ and $T>0$ such that
  if $(M,g)$ is compact Riemannian surface satisfying :
  \begin{itemize}
  \item $K_g\geq -1$
  \item $\diam(M,g)\leq D$
  \item $\frac{V}{2}\leq\vol(M,g)\leq V$
  \end{itemize}
  then the Ricci flow $(M,g(t))$ with (classic) initial condition
  $(M,g)$ exist at least for $t\in [0,T)$ and satifies :
  \begin{itemize}
  \item $-1\leq K_{g(t)}\leq \frac{\kappa}{t}$ for $t\in [0,T)$
  \item $\diam(M,g(t))\leq 2D$ for $t\in [0,T)$
  \item $\frac{V}{4}\leq\vol(M,g(t))\leq 2V$ for $t\in [0,T)$
  \item $d_{g(s)}-\kappa(\sqrt{t}-\sqrt{s})\leq d_{g(t)}\leq
    e^{\kappa(t-s)}d_{g(s)}$ for $0<s<t\leq T$.
  \end{itemize}
\end{thm}
\begin{rem}
  Note that in dimension 2 a lot of the arguments used by M. Simon to
  prove these estimates in dimension $3$ become very simple. Only the
  existence of $T$ and the $\kappa/t$ bound require a delicate blowup analysis.
\end{rem}
Using these estimates on each Ricci flow $(M_i,g_i(t))$ with classic
initial condition $(M_i,g_i)$, we have, using the compactness theorem
of Hamilton for flows, a subsequence which converges to a Ricci flow
$(M,g(t))$ defined for $t\in(0,T)$ which satisfies the estimates of
Theorem \ref{estiexist}. Using the estimate on the distances, we can
argue as in \cite{MS2009} to show that $(M,g(t))$ has $(X,d)$ as metric initial condition.

\section{Uniqueness of the conformal class}
\label{sec:uniq-conf-class}
In this section, we prove that the metric initial condition determines
the conformal class of the flow under the geometric estimates we have
assumed.
\begin{prop}\label{uniqconf}
  Let $(M_1,h_1)$ (resp. $(M_2,h_2)$) be compact Riemannian surfaces of
  constant curvature, $g_1(x,t)=w_1(x,t)h_1(x)$ (resp. $g_2(x,t)=w_2(x,t)h_2(x)$)
  a smooth Ricci flow on $M_1\times (0,T]$ (resp. $M_2\times (0,T]$).

  Assume that :
  \begin{enumerate}
  \item $-1\leq K_{g_1}(x,t)$ and $-1\leq K_{g_2}(x,t)$
  \item $(M_1,g_1(t))$ and $(M_2,g_2(t))$ have the same Alexandrov
    surface $(X,d)$ as metric initial condition.
  \end{enumerate}
  Then there exist a conformal diffeomorphism $\ph:(M_1,h_1)\to(M_2,h_2)$.
\end{prop}
Set $u_i(x,t)=\frac{1}{2}\log w_i(x,t)$. In the following lemmas, $u$
denotes either $u_1$ or $u_2$.
\begin{lem}
  When $t$ goes to $0$, $u(x,t)$ converges in $L^1$ norm to an integrable
  function $u_0(x)$.
\end{lem}
\begin{proof}
  Since $\partial_t u=-2K_g\leq 2$, we have that $u(x,t)-2t$ increases as
  $t$ decreases to $0$. This allows us to define the pointwise limit
  $u_0(x)$ of $u(t,x)$ as $t$ goes to $0$. If we fix $t_0>0$, this
  also gives us that, for $t\in (0,t_0)$, $u(x,t)\geq
  u(x,t_0)-2(t_0-t)$. Thus $u$ is uniformly bounded from
  below. Moreover, by Jensen's inequality :
  \[\exp\left(2\int_M u(x,t)\frac{dv_h}{\vol(M,h)}\right)\leq \int_M
  e^{2u(x,t)}\frac{dv_h}{\vol(M,h)}=\frac{\vol(M,g(t))}{\vol(M,h)}\]
  which gives that $u(.,t)$ is uniformly bounded in $L^1$, thus by
  monotone convergence, $u_0$ is in $L^1$ and the convergence is in
  $L^1$ norm.

\end{proof}
\begin{lem}
  $u_0$ belongs to the space $Pot(M,h)$ defined in the appendix.
\end{lem}
\begin{proof}
  The previous lemma shows that $u_0$ is an $L^1$ function.
  We just need to check that the distributional Laplacian of $u_0$ is a
  signed measure.

  To see this, we write, for a smooth function $\eta:M\to\mathbb{R}$ :
  \begin{align*}
    \int_M\eta(x)\Delta_{h}
    u(x,t)dv_h(x) &=\int_{M}\eta(x)(K_{h}-K_{g(t)}e^{2u(x,t)})dv_{h}\\
    &=\int_{M}\eta(x)K_{h}dv_{h}-\int_{M} \eta(x)d\omega_{g(t)}
  \end{align*}
  where $d\omega_{g(t)}=K_{g(t)}e^{2u(x,t)}dv_{h}$ is the
  curvature measure of $(M_i,g_i(t))$. By Theorem \ref{thmCVdist},
  since the distance $d_{g(t)}$ uniformly converges to to the
  distance $d$, the curvature measures weakly converges to the
  curvature measure of $(M,d)$ which we call $d\omega$. 
  We integrate by parts on the left side of the previous
  equality and let $t$ go to $0$, we get :
  \[\int_{M_i}u_{0}(x)\Delta_{h} \eta(x)dv_h=\int_{M}
  \eta(x)K_{h}dv_{h}-\int_{M_i}\eta(x)d\omega.\]
  This tells us that the distributional laplacian of $u_{0}$ is the
  measure $\mu=K_{h}dv_{h}-d\omega$.

\end{proof}
As in the appendix, we define a new distance on $M$ by
$d_0=d_{h,u_0}$. Since $(M,d)$ has curvature bounded from below, the
condition $d\mu^+(\{x\})<2\pi$ is satisfied (see Remark \ref{remnocusp}), and $d_0$ is a distance on $M$
whose induced topology is the usual manifold topology of $M$. 

\begin{lem}
  For any $x$ and $y$ in $M$, $d(x,y)=d_0(x,y)$.
\end{lem}
\begin{proof}
  For $t>0$, consider the curvature measures :
  \[d\omega_t=K_{g(t)}e^{2u(x,t)}dv_h.\]
  By Theorem \ref{thmCVdist}, the curvature measures weakly converge
  to the curvature 
  measure $d\omega$ of $(M,d)$. Moreover, since the curvature of
  $(M,d)$ is bounded from below by $-1$, $d\omega\geq
  -e^{2u_0}dv_h$. Set :
  \[d\mu_t=K_hdv_h-d\omega_t.\]
  As $t$ goes to $0$, $d\mu_t$ weakly converges to $d\mu$, since
  $d\mu$ is bounded from below by an integrable function, we ahve that
  $d\mu^+_t$ and $d\mu^-_t$ weakly converge to $d\mu^+$ and
  $d\mu^-$. We also have convergence of the volumes. We can then apply
  Theorem \ref{thmCVcurv} to get that $d_{g(t)}$ uniformly converges
  to $d_{h,u_0}$. This gives the
  claimed result.
\end{proof}
We will write $(M,e^{2u_0}h)$ for $M$ equiped with the distance $d_0$.

We are now ready to prove 
Proposition \ref{uniqconf} :
\begin{proof}[Proof (of Proposition \ref{uniqconf})]
  For each Ricci flow $(M_i,e^{2u_i(x,t)}h_i(x))$, we have constructed
  a $u_{i,0}(x)$ such that $(M_i,e^{2u_{i,0}}h_i)$ is isometric to
  $(X,d)$. Thus there exists an isometry $\ph$ from
  $(M_1,e^{2u_{1,9}(x)}h_2(x))$ to $(M_2,e^{2u_{2,0}(x)}h_2(x))$. Theorem
  \ref{uniqconfHub} exactly gives that $\ph$ is conformal form
  $(M,h_1)$ to $(M,h_2)$.
\end{proof}
\section{End of the proof}

Thanks to the results of the previous section, we can now assume that
$g_1(x,t)=w_1(x,t)h(x)$ and $g_2(x,t)=w_2(x,t)h(x)$ are two Ricci
flows on a surface $(M,h)$ with metric initial condition $(M,d)$
defined for $t$ in $(0,T]$ 

It is a standard fact that $w_1$ and $w_2$ satisfy the following equation of $M\times (0,T]$ :
\begin{equation}
\label{eq:1} 
  \ddt{w_i}=\Delta_h \log(w_i)-2K_h.
\end{equation}

The next lemma relate the metric initial condition with the behaviour
of $w_i$ as $t$ goes to $0$ :
\begin{lem}
  $w_i(.,t)dv_h$ weakly converges to the $2$-dimensional area measure
  $d\sigma$ associated with $d$.
\end{lem}
This is Theorem \ref{thmCVdist} in the appendix.

First we prove some estimates on $w_i$ :
\begin{lem}\label{estimw}
One can find $C>0$ depending on $K$, $w_1$ and $w_2$ only, such that :
\[ C e^t \leq w_i(x,t)\]
for all $x$ in $M\times (0,T]$.  
\end{lem}
\begin{proof}
  We set $w=w_1$, the proof is the same for $w_2$.
  We have :
\[ \partial_t g=\partial_t w h=-2K_g g=-2K_g w h\]
which gives $\partial_t w=-2K_g w$.
Using the geometric estimates on the curvature, we get :
\[\frac{\partial_t w}{w}\leq 2\]
Let $0<t_1<t_2<T$, compute at some fixed $x\in M$, then :
\[[\log(w(x,t))]_{t_1}^{t_2} \leq 2(t_2-t_1)\]
and :
\[\frac{w(x,t_2)}{w(x,t_1)} \leq e^{2(t_2-t_1)}\]
thus :
\[\frac{w(x,t_1)}{w(x,t_2)} \geq e^{2(t_1-t_2) }\] 
Let $t_1=t$ and $t_2>0$ be some fixed time in $(0,T)$ and use that $w(.,t_2)$ is smooth on $M$
compact, we get the required estimate.
\end{proof}
The weak convergence of $w_i(.,t)$ to $d\sigma$ is not really
pleasant to work with when dealing with uniqueness issues. In fact,
the following lemma shows that the convergence is strong in $L^1$.
\begin{lem}
  As $t$ goes to $0$, $w(.,t)$ converges in $L^1$ norm to a function $w_0$
  which satifies $w_0 dv_h=\mathcal{H}^2$.
\end{lem}
\begin{proof}
  Let $\tilde{w}(x,t)=e^{-2t}w(x,t)$, then :
  \[\partial_t \tilde{w}(x,t)=-2e^{-2t}w(x,t)+e^{-2t}\partial_t
  w(x,t)\]
  As in the proof of the previous lemma : $\partial_t w\leq 2 w$. So
  $\partial_t \tilde{w}\leq 0$ and $\tilde{w}(x,t)$ increases as $t$
  decreases to $0$. Let $w_0$ be the pointwise limit of
  $\tilde{w}(.,t)$ as $t$ goes to $0$. Since
  $\int_M\tilde{w}(x,t)dv_h=e^{-2t}\vol(M,g(t))$ is bounded, Lebesgue's monotone
  convergence theorem gives that $w_0$ is in $L^1$ and
  $\tilde{w}(.,t)$ (and $w(.,t)$)
  converges in $L^1$ norm to $w_0$. Since $L^1$ convergence implies
  weak convergence, $w_0 dv_h=d\sigma$.
\end{proof}

We now prove the uniqueness statement. 
\begin{prop}
  $w_1(x,t)=w_2(x,t)$ for any $x\in M$ and $t\in (0,T]$.
\end{prop}
\begin{proof}
  We will prove that for any smooth nonnegative function $\eta$ on
  $M$ and any $T'\in (0,T]$ :
  \[\int_M (w_1(x,T')-w_2(x,T'))\eta(x) dv_h(x)=0\]
  Let $\psi$ be a smooth function on $M\times (0,T']$ and $0<s<T'$, then :
  \begin{align*}
    &\int_M (w_2(x,T')-w_1(x,T'))\psi(x,t)dv_h-\int_M (w_2(x,s)-w_1(x,s))\psi(x,s)dv_h=\\
    &\int_s^{T'} \int_M (w_2(x,\tau)-w_1(x,\tau))(A(x,\tau)\Delta_h
    \psi(x,\tau)+\partial_t\psi(x,\tau))dv_h d\tau 
  \end{align*}
  where $A(x,\tau)=\frac{\log(w_2(x,\tau))-\log(w_1(x,\tau))}{w_2(x,\tau)-w_1(x,\tau)}$.
  Since $w_1$ and $w_2$ are smooth on $M\times (0,T]$, $A$ is smooth
  too. Moreover, by the mean value theorem and lemma \ref{estimw}, we
  have, for $(x,t)\in M\times (0,T]$ :
  \[\ph(t)\leq A(x,t)\leq\frac{1}{C_1}\]
  where $\ph$ is the positive continuous function defined by:
  \[\ph(t)=\inf_{x\in M}\min\left
    (\frac{1}{w_1(x,t)},\frac{1}{w_2(x,t)}\right )>0.\]

  We now choose $\psi$ to be the solution of the following backwark heat
  equation :
  \begin{equation*}
    \begin{cases}
      \ddt{\psi}(x,t)=-A(x,t)\Delta_h\psi(x,t)\\
      \psi(x,T')=\eta(x)
    \end{cases}
  \end{equation*}
  Thanks to the properties of $A$, $\psi$ is smooth on $M\times (0,T']$
  and the maximum principle shows that : $0\leq \psi(x,t)\leq \sup_{x\in
  M}\eta(x)$. we get :
\[ \int_M (w_2(x,T')-w_1(x,T'))\eta(x)dv_h=\int_M
(w_2(x,s)-w_1(x,s))\psi(x,s)dv_h\]
We now let $s$ go to $0$, since $w_1(.,s)-w_2(.,s)$ goes to $0$ in
$L^1$ norm and $\psi(x,s)$ is bounded, the right hand side of the
previous equality goes to $0$ and  :
\[\int_M (w_2(x,T')-w_1(x,T'))\eta(x)dv_h=0\]
Since this equality is true for any $\eta$ and any $T'>0$, we have
that $w_1$ and $w_2$ are equal almost everywhere, since these
functions are smooth, we get equality everywhere.
\end{proof}
\appendix

\section{Facts from the theory of Alexandrov surfaces}
\label{sec:facts-from-theory}

This appendix gathers the results from the theory of Alexandrov
surfaces with bounded integral curvature or curvature bounded
from below that have been used in the paper. All these results can be
found 
in the works of Alexandrov and Reshetnyak (see \cite{MR2193913},
\cite{MR0216434} and \cite{MR1263964}). A survey in a more
modern language can be found in \cite{2009arXiv0906.3407T}.

We use two notions of surfaces with special curvature properties in
this work. Our main objects of interest are compact surfaces with
curvature bounded from below by $-k$, which are surfaces with an intrinsic
metric $(X,d)$ whose geodesic triangles are ``fatter'' than those in
the complete simply-connected surface of constant curvature $-k$ (see
\cite{MR1835418}, chapter 4 and 10).

A wider class of surfaces is the class of surfaces with bounded integral
curvature in the sense of Alexandrov. The definition we give in the
next few lines
stays informal, precise definition can be found in \cite{MR0216434}
and \cite{MR1263964}. The excess of a geodesic triangle $T$ in an
intrinsic is defined by
$e(T)=(\alpha+\beta+\gamma)-\pi$ where $\alpha$, $\beta$ and $\gamma$
are the upper angles of $T$. A compact surface with an intrinsic
metric $(X,d)$ is said to have bounded integral curvature if there is
a constant $C$ such that for any finite family $(T_i)$ of disjoint
``nice enough'' triangles, $\sum_i |e(T_i)|\leq C$.

Compact Alexandrov surfaces with curvature bounded from below are
compact Alexandrov surfaces with bounded integral curvature, a proof
of this fact can be found in \cite{MR1643359}. Alexandrov surfaces
with bounded integral curvature have well defined notions of area and
curvature, which are measures on the surface (signed measure for the
curvature). In the case of compact smooth surfaces $(M,g)$, these
measures coincide with the usual notions of volume form $dv_g$ and
curvature measure $K_gdv_g$, see \cite{MR0216434}, chapters 5 and 8.

First we need a theorem on the approximation of compact Alexandrov
surfaces with curvature bounded from below by smooth surfaces :

\begin{thm}\label{approxalex}
  For any compact Alexandrov surface with curvature bounded from below
  by $k$ $(X,d)$, there exist a sequence of smooth compact Riemannian
  surfaces $(M_i,g_i)_{i\in\mathbb{N}}$ satisfying :
  \begin{itemize}
  \item $K_{g_i}\geq k$
  \item $\diam(M_i,g_i)\leq D$
  \item $\frac{V}{2}\leq \vol(M_i,g_i)\leq V$
  \end{itemize}
  which Gromov-Hausdorff converges to $(X,d)$.
\end{thm}

This theorem doesn't seem to have been explicitely stated before. When
$k=0$, it follows from the theorem of Alexandrov on the approximation
of convex surfaces by convex polyhedra, which is
proved in chapter 7, section 6 of \cite{MR2193913}, and the fact that convex
polyhedra can be approximated by smooth convex surfaces. When the curvature
bound is not $0$, one has to approximate the surface by polyhedra
whose faces are geodesic triangle in a space form of curvature $k$.

The next theorem shows that the curvature measure and the area measure
depend continuously
on the distance, this is Theorem 6, p. 240 and Theorem 9 p. 269 in \cite{MR0216434}.
\begin{thm}\label{thmCVdist}
  Le $(d_i)_{i\in\mathbb{N}}$ and $d$ be distances on a compact surface $M$ such that :
  \begin{itemize}
  \item $(M,d)$ and each of the $(M,d_i)$ are Alexandrov surfaces of
    bounded integral curvature.
  \item as functions on $M\times M$, the distances $d_i$ uniformly
    converges to $d$.
  \end{itemize}
  Then the curvature measures $d\omega_i$ of $(M,d_i)$ weakly
  converges to the curvature measure $d\omega$ of $(M,d)$, that is,
  for any continuous $\ph$ function on $M$ :
  \[\int_M \ph d\omega_i\xrightarrow{i\to\infty} \int_M \ph d\omega.\]
  Moreover, the area measure $d\sigma_i$ of $d_i$ weakly converge to
  the area measure $d\sigma$ of $d$.
\end{thm}

Our aim now is to present a partial converse of the previous
theorem. In the sequel, $h$ is a fixed smooth Riemannian metric on
$M$. We consider the space $Pot(M,h)$ of $L^1$ functions $u$ on $M$ whose
distributional laplacian with respect to $h$ is a signed measure $d\mu$
on $M$, we say that $u$ is the potential of $d\mu$. Such a $u$ is the
difference of two subharmonic functions and has a representative which
is well defined outside a set of Hausdorff dimension $0$ in $M$. 

The volume of $u$
is defined by $V(u)=\int_M e^{2u}dv_h$. Given a zero mass signed
measure $d\mu$ and $V>0$, $d\mu$ has a unique potential $u_{\mu,V}$ of
volume $V$. We will denote by
$d\mu=d\mu^+-d\mu^-$ the Jordan decomposition of $\mu$. Reshetnyak has
studied the non-smooth Riemannian metric $e^{2u}h$. We have
(\cite{MR1263964} Theorem 7.1.1, \cite{2009arXiv0906.3407T}
Proposition 5.3) :
\begin{thm}\label{subhdist}
  Let $u\in Pot(M,h)$ be a potential of $d\mu$. Assume that
  $d\mu^+(\{x\})<2\pi$ for any $x\in M$. Define :
  \[d_{h,u}(x,y)=\inf_{\gamma\in\Gamma(x,y)}\int_0^1
  e^{u(\gamma(\tau))}|\dot{\gamma}(\tau)|_hd\tau\]
  where $\Gamma(x,y)$ is the space of $C^1$ paths $\gamma$ from
  $[0,1]$ to $M$ with $\gamma(0)=x$ and $\gamma(1)=y$. Then $d_{h,u}$
  is a distance on $M$ such that $(M,d_{h,u})$ has bounded integral
  curvature. The curvature measure of this surface is given by :
  \[d\omega=K_h dv_h +d\mu.\]
\end{thm}

We are now ready to state the converse of Theorem
\ref{thmCVdist}. This is Theorem 7.3.1 in \cite{MR1263964}, see also
\cite{2009arXiv0906.3407T}, Theorem 6.2.
\begin{thm}\label{thmCVcurv}
  Let $(M,h)$ be a smooth Riemannian surface and $(d\mu_i^+)_{i\in\mathbb{N}}$
  $(d\mu_i^-)_{i\in\mathbb{N}}$ be two sequences of (nonnegative) measures which weakly
  converge to $d\mu^+$ and $d\mu^-$ and such that $d\mu_i(M)$ and
  $d\mu_i(M)$ are equal and bounded independently of $i$. 

  Let $V_i$ be a sequence of
  positive numbers converging to $V$. Let $u_i$ be the potential of
  $d\mu_i=d\mu_i^+-d\mu_i^-$ of volume $V_i$ and $u$ be the potential of
  $d\mu=d\mu^+-d\mu^-$ of volume $V$. 

  Assume that $d\mu(\{x\})<2\pi$
  for all $x\in M$. Then the distances $d_{h,u_i}$ uniformly converge
  as $i$ goes to infinity to the distance $d_{h,u}$.

\end{thm}
\begin{rem}\label{remnocusp}
  In the case of surfaces with curvature bounded from below, the
  condition $d\mu(\{x\})<2\pi$ is automatically fulfilled. In fact, it
  follows from the discussion on ``complete angles at a point'' in
  \cite{MR0216434} (Chapter 2 Section 5 and Chapter 4 Section 4) that
  if the curvature $d\omega(\{x\})$ of a point $x$ in $(M,d)$ is
  $2\pi$, then any two shortest paths $\gamma_1$ and $\gamma_2$ emanating of
  $x$ will make a $0$ angle at $x$. Since when the curvature is
  bounded from below this angle has to be greater than the comparison
  triangle, this is impossible.
\end{rem}

Then next theorem, due to Huber says that the distance $d_{h,u}$
determines the conformal class of $h$, see \cite{MR1263964} Theorem
7.1.3 or \cite{2009arXiv0906.3407T} Theorem 6.4.

\begin{thm}\label{uniqconfHub}
  Let $(M,h)$ and $(M',h')$ be two compact Riemannian surfaces, $u\in
  Pot(M,h)$ and $u'\in Pot(M',h')$. Assume $f$ is an isometry from
  $(M,d_{h,u})$ to $(M',d_{h',u'})$, the $f$ is a conformal
  diffeomorphism from $(M,h)$ to $(M',h')$.
\end{thm}


\begin{thebibliography}{{Sim}09a}

\bibitem[Ale06]{MR2193913}
A.~D. Aleksandrov.
\newblock {\em A. {D}. {A}leksandrov selected works. {P}art {II}}.
\newblock Chapman \& Hall/CRC, Boca Raton, FL, 2006.
\newblock Intrinsic geometry of convex surfaces, Edited by S. S. Kutateladze,
  Translated from the Russian by S. Vakhrameyev.

\bibitem[AZ67]{MR0216434}
A.~D. Aleksandrov and V.~A. Zalgaller.
\newblock {\em Intrinsic geometry of surfaces}.
\newblock Translated from the Russian by J. M. Danskin. Translations of
  Mathematical Monographs, Vol. 15. American Mathematical Society, Providence,
  R.I., 1967.

\bibitem[BBI01]{MR1835418}
Dmitri Burago, Yuri Burago, and Sergei Ivanov.
\newblock {\em A course in metric geometry}, volume~33 of {\em Graduate Studies
  in Mathematics}.
\newblock American Mathematical Society, Providence, RI, 2001.

\bibitem[GT11]{MR2832165}
Gregor Giesen and Peter~M. Topping.
\newblock Existence of {R}icci flows of incomplete surfaces.
\newblock {\em Comm. Partial Differential Equations}, 36(10):1860--1880, 2011.

\bibitem[Mac98]{MR1643359}
Yoshiroh Machigashira.
\newblock The {G}aussian curvature of {A}lexandrov surfaces.
\newblock {\em J. Math. Soc. Japan}, 50(4):859--878, 1998.

\bibitem[{Ram}11]{2011arXiv1109.5554R}
D.~{Ramos}.
\newblock {Smoothening cone points with Ricci flow}.
\newblock {\em ArXiv e-prints}, September 2011.

\bibitem[Res93]{MR1263964}
Yu.~G. Reshetnyak.
\newblock Two-dimensional manifolds of bounded curvature.
\newblock In {\em Geometry, {IV}}, volume~70 of {\em Encyclopaedia Math. Sci.},
  pages 3--163, 245--250. Springer, Berlin, 1993.

\bibitem[{Ric}11]{2011arXiv1111.0859R}
T.~{Richard}.
\newblock {Lower bounds on Ricci flow invariant curvatures and geometric
  applications}.
\newblock {\em ArXiv e-prints}, November 2011.

\bibitem[{Sim}09a]{MS2009}
M.~{Simon}.
\newblock {Ricci flow of non-collapsed 3-manifolds whose Ricci curvature is
  bounded from below}.
\newblock {\em ArXiv e-prints}, March 2009.

\bibitem[Sim09b]{MR2526789}
M.~{Simon}.
\newblock Ricci flow of almost non-negatively curved three manifolds.
\newblock {\em J. Reine Angew. Math.}, 630:177--217, 2009.

\bibitem[{Top}10]{2010arXiv1010.2795T}
P.~{Topping}.
\newblock {Uniqueness and nonuniqueness for Ricci flow on surfaces: Reverse
  cusp singularities}.
\newblock {\em ArXiv e-prints}, October 2010.

\bibitem[Tro09]{2009arXiv0906.3407T}
Marc Troyanov.
\newblock {Les surfaces \`a courbure int\'egrale born\'ee au sens
  d'Alexandrov.}
\newblock In {\em {Troyanov, Marc et al., G\'eom\'etrie discr\`ete,
  algorithmique, diff\'erentielle et arithm\'etique. Paris: Soci\'et\'e
  Math\'ematique de France. SMF Journ\'ee Annuelle 2009, 3-20 (2009).}}, 2009.

\end{thebibliography}
\end{document}